\documentclass[11pt]{article}
\usepackage{amsmath, amsfonts, amsthm, amssymb, color}
\usepackage{graphicx}
\usepackage{float}
\usepackage{verbatim}% for comment

\allowdisplaybreaks

\numberwithin{equation}{section}

\newtheorem{lemma}{Lemma}[section]

\newtheorem{prop}[lemma]{Proposition}
\newtheorem{thm}[lemma]{Theorem}
\newtheorem{cor}[lemma]{Corollary}
\theoremstyle{definition}

\newtheorem{conjecture}[lemma]{Conjecture}

\theoremstyle{remark}

\def\T{\mathbb{T}}
\def\R{\mathbb{R}}
\def\N{\mathbb{N}}

\def\bo{\boldsymbol{\omega}}

\numberwithin{equation}{section} \numberwithin{table}{section}

\title{New dimension bounds for $\alpha\beta$ sets}
\author{Simon Baker\\ \\
\emph{School of Mathematics,} \\ \emph{University of Birmingham,} \\ \emph{Birmingham,  B15 2TT, UK.} \\ Email: simonbaker412@gmail.com\\}

\date{\today}
\begin{document}
\maketitle

\begin{abstract}
In this paper we obtain new lower bounds for the upper box dimension of $\alpha\beta$ sets. As a corollary of our main result, we show that if $\alpha$ is not a Liouville number and $\beta$ is a Liouville number, then the upper box dimension of any $\alpha\beta$ set is $1$. We also use our dimension bounds to obtain new results on affine embeddings of self-similar sets.\\
%In this paper we give sufficient conditions for a self-similar measure $\mu,$ so that for $\mu$ almost every $x$ the sequence $(x^n)_{n=1}^{\infty}$ is uniformly distributed modulo one. 



\end{abstract}

\section{Introduction}
Let $\mathbb{T}:=\mathbb{R}/\mathbb{Z}$ denote the unit circle. Given $\alpha,\beta \in \R\setminus \mathbb{Q},$ a non-empty closed set $E\subset \T$ is called an $\alpha\beta$ set if for all $x\in E$ either $x+\alpha \mod 1\in E$ or $x+\beta \mod 1\in E$. A sequence $(x_n)_{n\geq 0}$ of points in $\T$ is called an $\alpha\beta$ orbit if for all $n\geq 0,$ either $x_{n+1}-x_{n}=\alpha \mod 1$ or $x_{n+1}-x_{n}=\beta \mod 1.$ Clearly any $\alpha\beta$ set contains an $\alpha\beta$ orbit. If $\alpha$ and $\beta$ are rationally dependent modulo one, i.e. there exists $n_1,n_2\in\mathbb{Z}$ such that $n_{1}\alpha+n_{2}\beta=0 \mod 1,$ then using the well known fact that orbits of irrational circle rotations are dense in $\T$ together with the Baire category theorem, it can be shown that every $\alpha\beta$ set has non-empty interior (see \cite[Theorem 1.5(i)]{FengXiong}). This observation naturally leads to the following question that was posed by Engelking in \cite{Eng}: Suppose that $\alpha$ and $\beta$ are rationally independent modulo one, do there exist nowhere dense $\alpha\beta$ sets? This question was answered by Katznelson in \cite{Kat}. He proved that if $\alpha$ and $\beta$ are rationally independent, then there do exist nowhere dense $\alpha\beta$ sets. Katznelson also proved that $\alpha\beta$ sets exist with arbitrarily small Hausdorff dimension. Interest in $\alpha\beta$ sets was renewed in a recent paper of Feng and Xiong \cite{FengXiong}. In this paper they connected $\alpha\beta$ sets and their higher dimensional analogues\footnote{Instead of just considering two elements $\alpha,\beta\in \mathbb{R}\setminus\mathbb{Q},$ one can consider $\alpha_{1},\ldots,\alpha_{n}\in\mathbb{R}\setminus\mathbb{Q}$ and then define appropriate analogues of $\alpha\beta$ sets and $\alpha\beta$ orbits.} to the existence of affine embeddings of self-similar sets. They proved that if $\alpha$ and $\beta$ are rationally independent then any $\alpha\beta$ set $E$ satisfies $E-E=\T$ or $E$ has non-empty interior. This result implies that if $\alpha$ and $\beta$ are rationally independent then any $\alpha\beta$ set $E$ satisfies $\underline{\dim}_{B}E\geq 1/2$. Further results on the dimension of $\alpha\beta$ sets and their higher dimensional analogues were obtained by Yu in \cite{Yu}. In this paper Yu conjectured that for rationally independent $\alpha$ and $\beta,$ any $\alpha\beta$ set $E$ satisfies $\dim_{B}E=1$\footnote{This conjecture was formulated in \cite{Yu} in terms of the lower box dimension. Our formulation is easily seen to be equivalent.}. In this paper we obtain new lower bounds for the upper box dimension of $\alpha\beta$ sets. These bounds depend upon the Diophantine properties of $\alpha$ and $\beta$. As a corollary of our main result, we give the first examples of $\alpha$ and $\beta$ satisfying the conclusion of Yu's conjecture where box dimension is replaced with upper box dimension. We conclude this introductory section by mentioning a paper of Chen, Wang, and Wen \cite{CWW} who considered random analogues of $\alpha\beta$ orbits. They proved that such sequences were almost surely uniformly distributed modulo one, and that the exponential sums along the orbit have square root cancellation.

\subsection{Statement of results} 
A well known theorem due to Dirichlet states that for any $x\in \R$ and $Q>1,$ there exists integers $p$ and $q$ such that $1\leq q\leq Q$ and $$\left|x-\frac{p}{q}\right|< \frac{1}{qQ}.$$ This implies that if $x$ is an irrational number, then the inequality $$\left|x-\frac{p}{q}\right|< \frac{1}{q^2}$$ has infinitely many solutions in integers $p$ and $q$. 
Given $\tau\geq 2$ we say that $x\in\R\setminus \mathbb{Q}$ is $\tau$-well approximable if there exists infinitely many $(p,q)\in \mathbb{Z}\times \mathbb{N}$ satisfying $$\left|x-\frac{p}{q}\right|< \frac{1}{q^{\tau}}.$$ We denote the set of $\tau$-well approximable numbers by $W(\tau)$. For $x\in \R\setminus \mathbb{Q}$ we define the exact order of $x$ to be $$\tau(x):=\sup\{\tau:x\in W(\tau)\}.$$ If $\tau(x)=\infty$ then we say that $x$ is a Liouville number. For $\tau\in[2,\infty)\cup \{\infty\}$ we denote the set of real numbers with exact order $\tau$ by $E(\tau)$. Equipped with these definitions we are now able to state the main result of this paper.

\begin{thm}
	\label{Main theorem}
Let $\tau_1,\tau_2\geq 2$ satisfy $2\tau_1<\tau_2 +2$ and suppose that $\alpha\in E(\tau_{1})$ and $\beta\in W(\tau_2).$ Then any $\alpha\beta$ orbit $(x_n)_{n\geq 0}$ satisfies $\overline{\dim}_{B}(\{x_n\})\geq 1- \frac{2(\tau_1-1)}{\tau_{2}}.$ 
\end{thm}Theorem \ref{Main theorem} immediately implies the following result.
%Note that Theorem \ref{Main theorem} only provides a better lower bound for the upper box dimension of an $\alpha\beta$ set then that which is already known when $\tau_1$ and $\tau_2$ also satisfy $4\tau_{1}<\tau_{2}+4$. 

\begin{cor}
	\label{Corollary}
	Assume that $\alpha$ is not a Liouville number and $\beta$ is a Liouville number. Then any $\alpha\beta$ orbit $(x_n)_{n\geq 0}$ satisfies $\overline{\dim}_{B}(\{x_n\})=1.$ 
\end{cor}
Since every $\alpha\beta$ set contains an $\alpha\beta$ orbit, we immediately see that suitable analogues of Theorem \ref{Main theorem} and Corollary \ref{Corollary} also hold for $\alpha\beta$ sets. We emphasise that the $\alpha$ and $\beta$ appearing in the statements of Theorem \ref{Main theorem} and Corollary \ref{Corollary} are rationally independent. This is because any rationally dependent $\alpha$ and $\beta$ must have the same exact order.

The rest of this paper is structured as follows. In Section \ref{prelimiaries} the relevant definitions from Fractal Geometry are given and we gather some useful results from the theory of continued fractions. In Section \ref{proof section} we prove Theorem \ref{Main theorem}. In Section \ref{applications} we apply Theorem \ref{Main theorem} to obtain a result on affine embeddings of self-similar sets.
\section{Preliminaries}
\label{prelimiaries}
\subsection{Dimension theory}
Let $F\subset \mathbb{R}^n$ and $s\geq 0$. Given $\delta>0$ we define $$\mathcal{H}_{\delta}^{s}(F):=\inf \left\{\sum_{i=1}^{\infty}Diam(U_i)^{s}:\{U_i\} \textrm{ is a }\delta\textrm{-cover of } F\right\}.$$ We define the $s$-dimensional Hausdorff measure of $F$ to be $$\mathcal{H}^{s}(F):=\lim_{\delta\to 0}\mathcal{H}_{\delta}^{s}(F).$$ The Hausdorff dimension of $F$ is given by $$\dim_{H}(F):=\inf\{s\geq 0:\mathcal{H}^{s}(F)=0\}=\sup\{s\geq 0:\mathcal{H}^{s}(F)=\infty\}.$$ Given a bounded set $F\subset \mathbb{R}^n,$ we let $N(F,r)$ denote the minimum number of closed balls of radius $r$ required to cover $F$. The upper box dimension of a bounded set $F$ is defined to be $$\overline{\dim_{B}}(F):=\limsup_{r\to 0}\frac{\log N(F,r)}{-\log r}.$$ The lower box dimension is defined similarly using liminf instead of limsup. When the lower and upper box dimensions coincide we refer to the common value as the box dimension and denote it by $\dim_{B}(F)$. For more on dimension theory and fractal sets we refer the reader to \cite{Fal}.

\subsection{Continued fractions}
Proofs of the properties stated below can be found in the books \cite{Bug} and \cite{Cas}.  

For any $x\in [0,1]\setminus \mathbb{Q},$ there exists a unique sequence $(a_{n})_{n\geq 1}\in\mathbb{N}^{\mathbb{N}}$ such that 
$$ x=\cfrac{1}{a_1+\cfrac{1}{a_2 +\cfrac{1}
		{a_3 + \cdots }}}.$$
We call the sequence $(a_n)$ the continued fraction expansion of $x$. Suppose $x$ has continued fraction expansion $(a_n),$ then for each $n\geq 1$ we let
$$ \frac{p_n}{q_n}:=\cfrac{1}{a_1+\cfrac{1}{a_2 +\cfrac{1}
		{a_3 + \cdots \cfrac{1}
			{a_n }}}}.$$ The fraction $p_n/q_n$ is called the $n$-th partial quotient of $x$. For any $x\in [0,1]\setminus \mathbb{Q},$ its sequence of partial quotients satisfies the following properties:
\begin{itemize}
	\item If we set $p_{-1}=1, q_{-1}=0, p_0=0, q_0=1$, then for any $n\geq 1$ we have 
\begin{align}
\label{property2}
p_n&=a_n p_{n-1}+p_{n-2}\\
q_n&=a_n q_{n-1}+q_{n-2}. \nonumber
\end{align}
	\item For any $n\geq 1$ we have 
	\begin{equation}
	\label{property1}\frac{1}{q_n(q_{n+1}+q_n)}<\left|x-\frac{p_n}{q_n}\right|<\frac{1}{q_nq_{n+1}}.
	\end{equation}

	\item If $q< q_{n+1}$ then \begin{equation}
	\label{property3}
	|qx-p|\geq |q_{n}x-p_n|
	\end{equation}for any $p\in\mathbb{Z}$.
\end{itemize} For $x\in \mathbb{R}$ we will on occasion use $\|x\|$ to denote the distance from $x$ to the nearest integer. 

We will use the following lemma in our proof of Theorem \ref{Main theorem}.
\begin{lemma}
	\label{interval lemma}
Let $x\in E(\tau)$ for some $\tau\geq 2.$ Then for any $\epsilon>0$, for all $q\in \R$ sufficiently large the interval $[q,q^{\tau+\epsilon-1}]$ contains the denominator of some partial quotient of $x$.  
\end{lemma}
\begin{proof}
Let $(q_n)_{n=1}^{\infty}$ denote the sequence of denominators of partial quotients of $x$ written in increasing order. Suppose $q>q_{1}$ is such that the interval $[q,q^{\tau+\epsilon-1}]$ does not contain the denominator of a partial quotient of $x$. Then let $n^{*}\geq 1$ be the unique integer satisfying 
\begin{equation}
\label{n equation}
q_{n^{*}}<q\quad \textrm{and}\quad q_{n^{*}+1}>q^{\tau+\epsilon-1}.
\end{equation}
 Equation \eqref{property2} implies that 
 \begin{equation}
 \label{a_n equation}
 q_{n+1}\leq 2a_{n+1}q_{n}
 \end{equation} for all $n\geq 1$. Combining \eqref{n equation} and \eqref{a_n equation} we have 
\begin{equation}
\label{a q equation}
2a_{n^{*}+1}\geq \frac{q_{n^{*}+1}}{q_{n^{*}}}>q^{\tau+\epsilon-2}>q_{n^{*}}^{\tau+\epsilon-2}.
\end{equation} Equations \eqref{property2} and \eqref{property1} imply that 
\begin{equation}
\label{placeholder}
\left|x-\frac{p_n}{q_n}\right|\leq \frac{1}{a_{n+1}q_{n}^{2}}
\end{equation}for all $n\geq 1$. It now follows from \eqref{a q equation} and \eqref{placeholder} that \begin{equation}
\label{finite solutions}\left|x-\frac{p_{n^{*}}}{q_{n^{*}}}\right|\leq  \frac{2}{q_{n^{*}}^{\tau+\epsilon}}.
\end{equation}Since $x\in E(\tau)$ inequality \eqref{finite solutions} can have only finitely many solutions. It follows that for all $q\in \mathbb{R}$ sufficiently large the interval $[q,q^{\tau+\epsilon-1}]$ must contain the denominator of a partial quotient of $x$.  
   	
\end{proof}
\section{Proof of Theorem \ref{Main theorem}}
\label{proof section}
Let $\alpha,\beta\in\mathbb{R}\setminus \mathbb{Q}$. To any $\alpha\beta$ orbit $(x_n)_{n\geq 0}$ we can associate a unique sequence $\boldsymbol{\omega}=(\omega_n)_{n\geq 1}\in \{\alpha,\beta\}^{\mathbb{N}}$ such that $$x_{n}-x_{n-1}=\omega_n \mod 1$$ for all $n\geq 1.$ Given $\bo\in\{\alpha,\beta\}^{\N}$ and $N\in\mathbb{N}$ we let $$|\bo|_{\alpha,N}:=\#\{1\leq n \leq N:\omega_{n}=\alpha\}$$ and 
$$|\bo|_{\beta,N}:=\#\{1\leq n \leq N:\omega_{n}=\beta\}.$$ The following proposition shows that if an $\alpha\beta$ orbit $(x_n)_{n\geq 0}$ is such that the quantities $|\bo|_{\alpha,N}$ and $|\bo|_{\beta,N}$ are not uniformly comparable then $\{x_n\}_{n\geq 0}$ is dense in $\T.$

\begin{prop}
	\label{easy prop}
Let $\alpha,\beta\in\mathbb{R}\setminus\mathbb{Q}$ and $(x_n)_{n\geq 0}$ be an $\alpha\beta$ orbit. Suppose that for any $C>1$ there exists infinitely many $N\in \N$ such that either $$|\bo|_{\alpha,N}\geq C\cdot |\bo|_{\beta,N}$$ or $$|\bo|_{\beta,N}\geq C\cdot |\bo|_{\alpha,N}.$$ Then $\{x_n\}$ is dense in $\T$.
\end{prop}
\begin{proof}
It follows from our hypothesis that the sequence $\bo$ must either contain arbitrarily long strings of consecutive $\alpha$ terms or consecutive $\beta$ terms. Since both $\alpha$ and $\beta$ are irrational, and any orbit of an irrational rotation is dense in $\T$, it follows that $\{x_n\}$ must also be dense in $\T$. 	
\end{proof}

\begin{prop}
	\label{hard prop}
	Let $\tau_1,\tau_2\geq 2$ satisfy $2\tau_1<\tau_2 +2$ and suppose that $\alpha\in E(\tau_{1})$ and $\beta\in W(\tau_2).$ Let $(x_n)_{n\geq 0}$ be an $\alpha\beta$ orbit for which there exists $C>1$ such that for all $N\in\mathbb{N}$ sufficiently large we have $$\frac{|\bo|_{\beta,N}}{C}\leq|\bo|_{\alpha,N}\leq C\cdot |\bo|_{\beta,N}.$$ Then $\overline{\dim}_{B}(\{x_n\})\geq 1- \frac{2(\tau_1-1)}{\tau_{2}} .$
\end{prop}
\begin{proof}
Without loss of generality we may assume that $\alpha,\beta\in [0,1]$. For the rest of the proof we fix $(x_n)_{n\geq 0}$ an $\alpha\beta$ orbit satisfying our hypothesis and let $\bo$ be the associated unique element of $\{\alpha,\beta\}^{\N}$. Without loss of generality we may further assume that $x_0=0.$ This means that for any $N\geq 1$ we have $$x_{N}=\alpha\cdot |\bo|_{\alpha,N}+\beta\cdot |\bo|_{\beta,N} \mod 1.$$ Notice that $|\bo|_{\alpha,N}+|\bo|_{\beta,N} =N$ for all $N\geq 1$. It follows from this observation and our hypothesis that there exists $C>1,$ not necessarily the same $C$ as in the statement of our proposition, such that 
\begin{equation}
\label{N comparable}
\frac{N}{C}\leq |\bo|_{\alpha,N}
\end{equation}for all $N$ sufficiently large.  

Let $\epsilon>0$ be arbitrary. Since $\beta\in W(\tau_2)$ there exists a sequence of reduced fractions $(p_l/q_l)_{l\geq 1}$ such that 
\begin{equation}
\label{good approximation}
\left|\beta-\frac{p_l}{q_l}\right|\leq \frac{1}{q_{l}^{\tau_2}}
\end{equation} for all $l\geq 1$. Without loss of generality we may assume that the sequence $(q_l)_{l=1}^{\infty}$ is strictly increasing. By Lemma \ref{interval lemma}, for all $l$ sufficiently large, there exists $q_{l}'$ the denominator of some partial quotient of $\alpha$ which satisfies $$q_{l}'\in \left[q_{l}^{\frac{\tau_2-2\epsilon}{2(\tau_{1}+\epsilon-1)}},q_{l}^{\frac{\tau_2-2\epsilon}{2}}\right].$$  For any $j\in \mathbb{N}$ we let $k_{j}$ denote the minimum of those $k\in\mathbb{N}$ satisfying $$\alpha j+\beta k \mod 1 \in \{x_n\}.$$ Equivalently $k_j$ is the smallest integer such that $|\bo|_{\alpha,j+k_j}=j$. Notice that for any $N\in\mathbb{N}$, if $1\leq j\leq |\bo|_{\alpha,N}$ then we must have $k_{j}< N.$ For all $l$ sufficiently large so that $q_{l}'$ is well defined,  we let $$W(l,p):=\{1\leq j\leq |\bo|_{\alpha,q_{l}'}:k_{j}=p \mod q_{l}\}$$ for each $0\leq p\leq q_{l}-1$. By the pigeonhole principle and \eqref{N comparable}, for all $l$ sufficiently large there exists $0\leq p'\leq q_{l}-1$ such that 
\begin{equation}
\label{count}
\#W(l,p')\geq \frac{q_l'}{Cq_{l}}.
\end{equation} We now set out to prove that the elements of $\{x_n\}$ corresponding to the elements of $W(l,p')$ are well separated. Observe now that for any distinct $j,j'\in W(l,p')$ we have 
\begin{equation}
\label{split}
\|(\alpha j +\beta k_j)-(\alpha j'+\beta k_{j'})\|\geq \underbrace{\|\alpha(j-j')\|}_{(1)}-\underbrace{\|\beta(k_j-k_{j'}\|}_{(2)}.
\end{equation}
We now show how (1) can be bounded from below and (2) can be bounded from above. Notice that $j-j'$ is a non-zero integer satisfying $|j-j'|<q_{l}'$. Combining \eqref{property1} and \eqref{property3} it follows that
\begin{equation}
\label{lower bound}
\|\alpha(j-j')\|\geq \frac{1}{2q_{l}'}.
\end{equation}
Now focusing on $(2),$ let $d_j,d_{j'}\in \mathbb{N}$ be such that $k_{j}=d_jq_l+p'$ and $k_{j'}=d_{j'}q_{l}+p'$. Then we have 
\begin{align}
\label{upper bound}
\|\beta(k_{j}-k_{j'})\|&\leq \left\|\left(\beta-\frac{p_{l}}{q_l}\right)(k_{j}-k_{j'})\right\|+\left\|\frac{p_l}{q_l}(k_{j}-k_{j'})\right\|\nonumber\\
&\leq \frac{q_{l}'}{q_{l}^{\tau_2}}+\left\|\frac{p_l}{q_l}(d_{j}q_{l}-d_{j'}q_l)\right\|\nonumber\\
&=\frac{q_{l}'}{q_{l}^{\tau_2}}+\|p_l(d_{j}-d_{j'})\|\nonumber\\
&=\frac{q_{l}'}{q_{l}^{\tau_2}}\nonumber\\
&\leq \frac{1}{q_{l}^{\tau_2/2}}.
\end{align} In the second line in the above we have used \eqref{good approximation} and the inequality $|k_j-k_{j'}|<q_{l'}.$ This inequality follows from the fact that $k_j$ and $k_{j'}$ are integers satisfying $0\leq k_j, k_{j'}<q_{l'}$. In the final line we used that $q_{l}'\leq q_{l}^{\frac{\tau_2-2\epsilon}{2}}.$ Substituting \eqref{lower bound} and \eqref{upper bound} into \eqref{split} we have 
\begin{equation}
\label{separation}\|(\alpha j +\beta k_j)-(\alpha j'+\beta k_{j'})\|\geq \frac{1}{2q_{l}'}- \frac{1}{q_{l}^{\tau_{2}/2}}.
\end{equation}
Since $q_{l}'\leq q_{l}^{\frac{\tau_2-2\epsilon}{2}},$ for $l$ sufficiently large we have $$\frac{1}{2q_{l}'}- \frac{1}{q_{l}^{\tau_{2}/2}}\geq \frac{1}{2q_{l}'}\left(1-\frac{2q_{l}'}{q_{l}^{\tau_2/2}}\right)\geq \frac{1}{2q_{l}'}\left(1-\frac{2}{q_{l}^{\epsilon}}\right)\geq \frac{1}{4q_{l}'}.$$ Using this lower bound in \eqref{separation}, it follows that for $l$ sufficiently large, for any distinct $j,j'\in W(l,p')$ we have $$\left\|(\alpha j +\beta k_{j})-(\alpha j'+\beta k_{j'})\right\|\geq \frac{1}{4q_{l}'}.$$ Therefore for any $l$ sufficiently large we require at least $\#W(l,p')$ closed balls of radius $(10q_{l}')^{-1}$ to cover $\{x_n\}$. Using the lower bound for $\#W(l,p')$ provided by \eqref{count} and the inequality $q_{l}'\geq q_{l}^{\frac{\tau_2-2\epsilon}{2(\tau_{1}+\epsilon-1)}},$ we have
\begin{align*}
\overline{\dim}_{B}(\{x_n\})=\limsup_{r\to 0}\frac{\log N(\{x_n\},r)}{-\log r}&\geq \limsup_{l\to\infty} \frac{\log q_{l}'/Cq_{l}}{\log 10q_{l}'}\\
&\geq 1-\liminf_{l\to\infty} \frac{\log q_{l}}{\log q_{l'}}\\
&\geq 1 - \frac{2(\tau_{1}+\epsilon-1)}{\tau_{2}-2\epsilon}.
\end{align*}Since $\epsilon$ was arbitrary we may conclude 
$$\overline{\dim}_{B}(\{x_n\})\geq 1- \frac{2(\tau_1-1)}{\tau_{2}}.$$
\end{proof} Since any dense subset of $\T$ has upper box dimension $1$, Propositions \ref{easy prop} and \ref{hard prop} together imply Theorem \ref{Main theorem}.
\section{Applications to embeddings of self-similar sets}
\label{applications}

We call a map $\varphi:\mathbb{R}^d\to\mathbb{R}^d$ a similarity if there exists $r\in(0,1), t\in \mathbb{R}^d,$ and a $d\times d$ orthogonal matrix $O$ such that $\varphi=r\cdot O +t.$ For our purposes, we call a finite set of similarities $\Phi=\{\varphi_i\}_{i\in I}$ an iterated function system or IFS for short. A well known result due to Hutchinson \cite{Hut} states that for any IFS $\Phi,$ there exists a unique non-empty compact set $F\subset\mathbb{R}^d$ satisfying $$F=\bigcup_{i\in I}\varphi_{i}(F).$$ We call $F$ the self-similar set of $\Phi$. Many well known fractal sets, such as the middle third Cantor set and the von-Koch curve, can be realised as self-similar sets for appropriate choices of IFS. If $\varphi_{i}(F)\cap \varphi_{j}(F)=\emptyset$ for all $i\neq j$ then we say that $\Phi$ satisfies the strong separation condition. We say that $\Phi$ satisfies the open set condition if there exists a non-empty bounded open $O\subset \mathbb{R}^d$ such that $\varphi_{i}(O)\subset O$ for all $i\in I$ and $\varphi_{i}(O)\cap \varphi_{j}(O)=\emptyset$ for all $i\neq j$.

Let $A,B\subset\mathbb{R}^d$. We say that $A$ can be affinely embedded into $B$ if there exists a map $f:\mathbb{R}^d\to\R^d$ of the form $f(x)=Mx +a$ for some invertible matrix $M$ and $a\in\mathbb{R}^{d}$ which satisfies $f(A)\subset B$. It is an interesting problem to determine when one self-similar set can be affinely embedded inside of another. This problem was first studied in \cite{FHR}. It is reasonable to expect that if a self-similar set can be affinely embedded inside of another self-similar set which is totally disconnected, then the underlying contraction ratios should exhibit some arithmetic dependence. With this in mind the authors of \cite{FHR} formulated the following conjecture.

\begin{conjecture}
	\label{Embed conjecture}
Suppose that $E,F$ are two totally disconnected non-trivial self-similar sets in $\R^d$, generated by IFSs $\Phi=\{\varphi_{i}\}_{i\in I}$ and $\Psi=\{\psi_j\}_{j\in J}$ respectively. Let $r_i, r_j'$ denote the contraction ratios of $\varphi_i$ and $\psi_j$ respectively. Suppose that $F$ can be affinely embedded into $E$. Then for each $j\in J$ there exists non-negative rational numbers $t_{i,j}$ such that $r_{j}'=\prod_{i\in I}r_{i}^{t_{i,j}}$. In particular, if $r_i=r$ for all $i\in I,$ then $\log r_{j}'/\log r\in \mathbb{Q}$ for all $j\in J$.
\end{conjecture}
Conjecture \ref{Embed conjecture} was studied in \cite{Alg1, Alg2, FHR, FengXiong, Shm, Wu}. In \cite{FHR} it was shown that Conjecture \ref{Embed conjecture} is true if we also assume that $\Phi$ satisfies the strong separation condition, $r_i=r$ for all $i\in I,$ and $\dim_{H}(E)<1/2$. Similar results were obtained in \cite{FengXiong} without the assumption $r_i=r$ for all $i\in I$. These results come at the cost that $\dim_{H}(E)$ is required to satisfy a stricter upper bound. In particular, the results of \cite{FengXiong} imply that when $\Phi$ consists of two similarities then Conjecture \ref{Embed conjecture} is true if we also assume that $\Phi$ satisfies the strong separation condition and $\dim_{H}(E)<1/4$. Shmerkin and Wu obtained much stronger results when $d=1$. Shmerkin in \cite{Shm} showed that Conjecture \ref{Embed conjecture} is true under the additional assumptions that $d=1$, $\Phi$ satisfies the open set condition, $r_i=r$ for all $i\in I,$  and $\dim_{H}(E)<1$. Wu in \cite{Wu} obtained the same result as Shmerkin but required the stronger assumption that $\Phi$ satisfies the strong separation condition. 

Our main result in this direction is the following theorem.
\begin{thm}
	\label{embed theorem}
Let $\Phi=\{\varphi_i\}_{i\in I}$ and $\Psi=\{\psi_j\}_{j\in J}$ be two IFSs satisfying the following properties:
	\begin{enumerate}
		\item $\Phi$ satisfies the strong separation condition.
		\item There exists $r_1,r_2\in(0,1)$ and $I_{1},I_{2}\subset I$ such that $\Phi=\{\varphi_{i,1}=r_{1}O_{i,1}+t_{i,1}\}_{i\in I_{1}} \cup\{\varphi_{i,2}=r_{2}O_{i,2}+t_{i,2}\}_{i\in I_{2}}.$
		\item There exists $j^{*}\in J$ such that:
		\begin{enumerate}
			\item $\psi_{j^{*}}=r_{j^*}'I_{d}+t_{j^*}.$
			\item There exists $\tau_1,\tau_2\geq 2$ satisfying $2\tau_1<\tau_2+2$ and $$-\frac{\log r_{1}}{\log r_{j^*}'}\in E(\tau_1)\quad \text{and} \quad -\frac{\log r_{2}}{\log r_{j^*}'}\in W(\tau_2).$$
		\end{enumerate}
	\end{enumerate} Then if $\dim_{H}(E)< \frac{1}{2}\left(1-\frac{2(\tau_1-1)}{\tau_2}\right)$ then $F$ cannot be affinely embedded into $E$.
\end{thm}
%\begin{thm}
%	\label{embed theorem}
%Let $\Phi$ and $\Psi$ be two IFSs of the form $\Phi=\{\varphi_{i,1}=r_{1}O_{i,1}+t_{i,1}\}_{i\in I_{1}} \cup\{\varphi_{i,2}=r_{2}O_{i,2}+t_{i,2}\}_{i\in I_{2}}$ and $\Psi=\{\psi_{j}=r_{j}I_{d}+t_{j}\}.$ We assume that $\Phi$ satisfies the open set condition and $$\frac{\log r_{1}}{\log r_{j}^*}\in E(\tau_1)\quad \text{and} \quad \frac{\log r_{2}}{\log r_{j}^*}\in W(\tau_2)$$ for some $j^*\in J$ and $\tau_1,\tau_2\geq 2$ satisfying $2\tau_1<\tau_2+2.$ If $\dim_{H}(E)< \frac{1}{2}\left(1-\frac{2(\tau_1-1)}{\tau_2}\right)$ then $F$ cannot be affinely embedded into $E$.
%\end{thm}
Theorem \ref{embed theorem} has the following corollary.

\begin{cor}
	\label{embed corollary}
Let $\Phi=\{\varphi_i\}_{i\in I}$ and $\Psi=\{\psi_j\}_{j\in J}$ be two IFSs satisfying the following properties:
\begin{enumerate}
	\item $\Phi$ satisfies the strong separation condition.
	\item There exists $r_1,r_2\in(0,1)$ and $I_{1},I_{2}\subset I$ such that $\Phi=\{\varphi_{i,1}=r_{1}O_{i,1}+t_{i,1}\}_{i\in I_{1}} \cup\{\varphi_{i,2}=r_{2}O_{i,2}+t_{i,2}\}_{i\in I_{2}}.$
	\item There exists $j^{*}\in J$ such that:
	\begin{enumerate}
		\item $\psi_{j^{*}}=r_{j^*}I_{d}+t_{j^*}.$
		\item $-\frac{\log r_{1}}{\log r_{j^*}'}$ is not a Liouville number and $-\frac{\log r_{2}}{\log r_{j^*}'}$ is a Liouville number.
	%	\item There exists $\tau_1,\tau_2\geq 2$ satisfying $2\tau_1<\tau_2+2$ and $$\frac{\log r_{1}}{\log r_{j}^*}\in E(\tau_1)\quad \text{and} \quad \frac{\log r_{2}}{\log r_{j}^*}\in W(\tau_2).$$
	\end{enumerate}
\end{enumerate} Then if $\dim_{H}(E)< \frac{1}{2}$ then $F$ cannot be affinely embedded into $E$.
\end{cor}
We emphasise that property $2.$ in the statement of Theorem \ref{embed theorem} and Corollary \ref{embed corollary} means that the IFS $\Phi$ consists of similarities whose contraction ratios are either $r_1$ or $r_2.$ Property 3a. means that the similarity $\psi_{j^*}$ has the identity matrix as its rotation component. One of the strengths of Theorem \ref{embed theorem} and Corollary \ref{embed corollary} is that they provide information when the elements of $\Phi$ have different contraction ratios. Most results in this area have the additional assumption that the elements of $\Phi$ have the same contraction ratio (see \cite{Alg1,Alg2,FHR,Shm,Wu}). Moreover, at the cost of an additional Diophantine condition and rotation assumption, these statements allows us to weaken the dimension assumption $\dim_{H}(E)<1/4$ that was needed in the work of Feng and Xiong \cite{FengXiong}.

Our proof of Theorem \ref{embed theorem} is essentially the same argument as one that is used in the proof of Theorem 1.2 from \cite{FengXiong}, apart from a few minor changes. We include the details of this proof for completion.

\begin{proof}[Proof of Theorem \ref{embed theorem}]
Let $\Phi$ and $\Psi$ be two IFSs satisfying the hypothesis of Theorem \ref{embed theorem}. Suppose that $F$ can be affinely embedded into $E$. Let $M$ be an invertible matrix and $a\in \mathbb{R}^d$ be such that 
\begin{equation}
\label{embed equation}
M(F)+a\in E.
\end{equation} We will now set out to prove that $$\dim_{H}(E)\geq 	\frac{1}{2}\left(1-\frac{2(\tau_1-1)}{\tau_2}\right)$$ and thus conclude our theorem.

Let $x_{j^*}\in F$ denote the unique point satisfying $\psi_{j^*}(x_{j^*})=x_{j^*}.$ Clearly $x_{j^*}\in \psi_{j^*}^n(F)$ for all $n\in \mathbb{N}$. Let $y_{j}^{*}$ be given by $$y_{j^*}:=Mx_{j^*}+a.$$ By \eqref{embed equation} we know that $y_{j}^*\in E$. Therefore there exists a sequence $(i_m)\in I^{\N}$ such that $y_{j^*}=\lim_{m\to\infty} \varphi_{i_1\ldots i_m}(0).$ Here and throughout we use $\varphi_{i_1\ldots i_m}$ to denote the concatenation $\varphi_{i_1}\circ \cdots \circ \varphi_{i_m}$ and $r_{i_1\ldots i_m}$ to denote the product $\prod_{l=1}^{m}r_{i_l}$. Our point $y_{j^*}$ satisfies $y_{j^*}\in  \varphi_{i_1\ldots i_m}(E)$ for all $m\in \mathbb{N}$. It therefore follows from the above that
\begin{equation}
\label{intersection}
(M(\psi_{j^*}^n(F))+a)\cap  \varphi_{i_1\ldots i_m}(E)\neq \emptyset
\end{equation} for all $n,m\geq 0$. Because $\Phi$ satisfies the strong separation condition we have
\begin{equation*}
\label{separation parameter}
c:=\inf_{i\neq i'}d(\varphi_{i}(E),\varphi_{i'}(E))>0.
\end{equation*} It is also the case that for each $m\in\mathbb{N}$ we have 
\begin{equation}
\label{level n separation}
d( \varphi_{i_1\ldots i_m}(E),E\setminus  \varphi_{i_1\ldots i_m}(E))\geq cr_{i_1\ldots i_{m-1}}.
\end{equation}It therefore follows from \eqref{intersection} and \eqref{level n separation} that 
\begin{equation}
\label{inclusion}M(\psi_{j^*}^n(F))+a\subset  \varphi_{i_1\ldots i_m}(E)\quad \text{whenever} \quad Diam(M(\psi_{j^*}^n(F)))<cr_{i_1\ldots i_{m-1}}.\end{equation}  For $m\geq 1$ define 
\begin{equation}
\label{s parameter}
s_{m}:=\min\left\{n\in\mathbb{N}:M(\psi_{j^*}^n(F))+a\subset \varphi_{i_1\ldots i_{m}}(E)\right\}.
\end{equation} It follows from \eqref{inclusion} that $s_{m}<\infty$. 

We introduce the notation:
\begin{align*}
\|M\|:&=\max\{|Mv|:|v|=1\}\\
\|M\|':&=\min\{|Mv|:|v|=1\}.\\
\end{align*} By \eqref{s parameter} we have $$\|M\|'(r_{j^*}')^{s_{m}}Diam(F)\leq Diam(M(\psi_{j^*}^{s_m}(F)))\leq Diam(\varphi_{i_1\ldots i_{m}}(E))\leq Diam(E)\cdot r_{i_1\ldots i_{m}}.$$ Therefore 
\begin{equation}
\label{upperbound2}
\frac{(r_{j^*}')^{s_m}}{r_{i_1\ldots i_{m}}}\leq \frac{Diam(E)}{\|M\|'Diam(F)}
\end{equation} for all $m\geq 1$. Similarly we have 
\begin{equation}
\label{lowerbound2}
\frac{(r_{j^*}')^{s_m}}{r_{i_1\ldots i_{m}}}\geq \frac{c\cdot r_{j^*}'}{\|M\|Diam(F)\max\{r_1,r_2\}}
\end{equation}when $s_{m}\geq 1$. Equation \eqref{lowerbound2} follows because if it were to fail then we would have 
$$Diam(M(\psi_{j^{*}}^{s_{m}-1}(F)))\leq \|M\| (r_{j^*}')^{s_{m}-1}Diam(F)<\max\{r_1,r_2\}^{-1}c\cdot r_{i_1\ldots i_m}\leq c\cdot r_{i_1\ldots i_{m-1}}.$$ Which by \eqref{inclusion} would imply $M(\psi_{j^{*}}^{s_{m}-1}(F))+a\subset \varphi_{i_1,\ldots,i_m}(E).$ This would contradict the definition of $s_{m}$. 

It follows from the definition of $s_{m}$ that $$\varphi_{i_1\ldots i_m}^{-1}(M(\psi_{j^{*}}^{s_{m}}(F))+a)\subset E.$$ Letting $Q_m=(O_{i_1}\circ\cdots \circ O_{i_m})^{-1}\circ M$ we have $$r_{i_1\ldots i_m}^{-1}\cdot (r_{j^*}')^{s_m}\cdot Q_{m}(F)+a_m\subset E$$ for some $a_{m}\in \R^d$. Here we used the fact that the rotation component for $\psi_{j^{*}}$ is the identity matrix. Therefore
\begin{equation}
\label{differences}
r_{i_1\ldots i_m}^{-1}\cdot (r_{j^*}')^{s_m}\cdot Q_{m}(F-F)\subset E -E
\end{equation} for $m\geq 1$. Let $v\in F-F$ be a non-zero vector. Such a vector must exists because $F$ is non-trivial. Then by \eqref{differences} we have
\begin{equation}
\label{blah}
r_{i_1\ldots i_m}^{-1}\cdot (r_{j^*}')^{s_m}\cdot Q_{m}v\subset E -E
\end{equation}for all $m\geq 1$. Using the fact that $Q_{m}$ is the composition of some orthogonal matrices with $M$, we see that by taking norms of both sides in \eqref{blah} we have
\begin{equation}
\label{inclusion2}
r_{i_1\ldots i_m}^{-1}\cdot (r_{j^*}')^{s_m}\cdot |Mv|\in\left\{|x-y|:x,y\in E\right\}
\end{equation}for all $m\geq 1$. Let $$U:=\left\{|x-y|:x,y\in E\right\}$$ and $$V:=\left\{r_{i_1\ldots i_m}^{-1}(r_{j^*}')^{s_m}|Mv|:m\geq 1\right\}.$$Consider the map $$f:\left[ \frac{c\cdot r_{j^*}'\cdot|Mv|}{\|M\|Diam(F)\max\{r_1,r_2\}},\frac{Diam(E)\cdot |Mv|}{\|M\|'Diam(F)}\right]\to \T\quad \text{given by}\quad f(x)=\frac{\log x}{\log r_{j^*}'}\mod 1.$$ The map $f$ is Lipschitz. It now follows from \eqref{upperbound2}, \eqref{lowerbound2}, and the well known fact that Lipschitz maps cannot increase the upper box dimension (see \cite{Fal}) that 
\begin{equation*}
\overline{\dim}_{B}f(V)\leq \overline{\dim}_{B}(V)\leq \overline{\dim}_{B}(U)\leq  \overline{\dim}_{B}(E-E)\leq \overline{\dim}_{B}(E\times E)=2\dim_{H}(E).
\end{equation*}Therefore 
\begin{equation}
\label{dimension inequalities}
\frac{\overline{\dim}_{B}f(V)}{2}\leq \dim_{H}(E).
\end{equation}Notice that for any $m\geq 1$ $$f\left(r_{i_1\ldots i_{m+1}}^{-1}r_{j^{*}}^{s_{m+1}}|Mv|\right)-f\left(r_{i_1\ldots i_m}^{-1}r_{j^{*}}^{s_m}|Mv|\right)=-\frac{\log r_{m+1}}{\log r_{j^*}'}\mod 1.$$ By property 2. the IFS $\Phi$ consists of similarities with contraction ratios equal to $r_1$ or $r_2$. Therefore $f(V)$ is an $\alpha\beta$ orbit for $\alpha=-\frac{\log r_1}{\log r_{j}^{*}}$ and $\beta=-\frac{\log r_2}{\log r_{j}^{*}}.$ Applying Theorem \ref{Main theorem} and \eqref{dimension inequalities} we have $$\dim_{H}(E)\geq \frac{1}{2}\left(1-\frac{2(\tau_1-1)}{\tau_2}\right).$$  This completes our proof.

\end{proof}

\end{document}